\documentclass[11pt,a4paper]{amsart}
\usepackage{amssymb,xspace}
\usepackage{amstext}
\theoremstyle{plain}
\usepackage{amsbsy,amssymb,amsfonts,latexsym}
\usepackage[utf8]{inputenc}
\usepackage{xcolor}
\usepackage{color}

\usepackage{enumerate}
\usepackage{graphics}

\newcommand{\aaa}{\alpha}

\newcommand{\ddd}{\delta}

\newcommand{\defeq}{{\buildrel {\rm def}\over =}}

\newcommand{\ds}{\displaystyle}

\newcommand{\ess}{\emptyset}

\newcommand{\fff}{{\varphi}}

\newcommand{\lf}{\lfloor}

\newcommand{\oo}{\infty}
\newcommand{\ooo}{\omega}
\newcommand{\OOO}{\Omega}

\newcommand{\ppp}{\psi}

\newcommand{\rf}{\rfloor}
\newcommand{\RR}{{\Rightarrow}}

\newcommand{\sm}{\setminus}
\newcommand{\sse}{\subset}

\renewcommand{\ggg}{\gamma}

\renewcommand{\lll}{\lambda}

\newcommand{\veps}{\varepsilon}

\def\QQ{\mathbb Q}

\def\RR{\mathbb R}

\def\NN{\mathbb N}
\def\ZZ{\mathbb Z}
\def\TT{\mathbb T}

\def\dimh{\dim_{\mathcal H}}

\def\lip{\textrm{Lip}}

\newtheorem{theorem}{Theorem}[section]

\newtheorem{lemma}[theorem]{Lemma}

\newtheorem{corollary}[theorem]{Corollary}

{\theoremstyle{definition}}
{\theoremstyle{definition}}

{\theoremstyle{definition}\newtheorem{example}[theorem]{Example}}

{\theoremstyle{definition}}

{\theoremstyle{definition}}

{\theoremstyle{definition}\newtheorem{remark}[theorem]{Remark}}

\newtheorem{question}[theorem]{Question}

\date{\today}
\title{Fast and slow points of Birkhoff sums}
\author{Fr\'ed\'eric Bayart, Zolt\'an Buczolich, Yanick Heurteaux}
\address{Universit\'e Clermont Auvergne, LMBP, UMR 6620 - CNRS, Campus des C\'ezeaux, 3 place Vasarely, TSA 60026, CS 60026 F-63178 Aubi\`ere Cedex, France.
}
\address{
Department of Analysis\\
ELTE E\"otv\"os Lor\'and
University\\ P\'azm\'any P\'eter S\'et\'any 1/c\\
 1117 Budapest, Hungary
}
\email{Frederic.Bayart@uca.fr, buczo@cs.elte.hu, Yanick.Heurteaux@uca.fr}
\thanks{The first and the third author were partially supported by the grant ANR-17-CE40-0021 of the French National Research Agency ANR (project Front).
The second author was supported by the Hungarian National Research, Development and Innovation Office--NKFIH, Grant  124003. He also thanks the R\'enyi Institute where he was
a visiting researcher during the academic year 2017-18.
}



\begin{document}

\begin{abstract}
  We investigate the growth rate of the Birkhoff sums $S_{n,\alpha}f(x)=\sum_{k=0}^{n-1}f(x+k\alpha)$,
where $f$ is a continuous function with zero mean defined on the unit circle $\mathbb T$ and $(\alpha,x)$ is
a ``typical'' element of $\TT^2$. The answer depends on the meaning given to the word ``typical''. Part of the work will be done in a more general context. 
\end{abstract}

\maketitle

\section{Introduction}

Let $\TT=\RR / \ZZ$ be the unit circle and let $\alpha\in\RR\backslash\QQ$ be irrational. 
Denote by $C_0(\TT)$,
the set of continuous functions on $\TT$ with zero mean, and by $S_{n,\alpha}f(x)$ the $n$-th Birkhoff sum, $S_{n,\alpha}f(x)=\sum_{k=0}^{n-1}f(x+k\alpha)$.
The rotation $R_\alpha:x\mapsto x+\alpha$
defines a uniquely ergodic transformation on $\TT$ with respect to the (normalized) Lebesgue measure $\lambda$. Hence for all $f\in\mathcal C_0(\TT)$
we know that $S_{n,\alpha}f(x)=o(n)$ for all  $x\in\TT$. The main purpose of this paper is to investigate the typical growth of $S_{n,\alpha}f(x)$.

There are several ways to understand this problem. We can fix $\alpha\in\RR\backslash\QQ$  (resp. $x\in\TT$)  and ask for the behaviour of $S_{n,\alpha}f(x)$ for $f$ in a generic subset 
of $\mathcal C_0(\TT)$ and for a typical $x\in \TT$  (resp. for a typical $\alpha\in\TT$). We can also consider it as a problem of two variables and ask for the behaviour of $S_{n,\alpha}f(x)$ for $f$
in a generic subset of $\mathcal C_0(\TT)$ and for a typical $(\alpha,x)\in\TT^2$. 
  There are also several ways to understand the word ``typical''. We can look for a residual set of the parameter space or for a set of full  Lebesgue measure. 

  We shall try to put this in a general context. If we fix $\alpha\in\RR\backslash\QQ$, then we consider the Birkhoff sums associated to a uniquely ergodic transformation 
on the compact metric space $\TT$. Hence, let us fix $\Omega$ an infinite compact metric space and $T: \Omega\to \Omega$ an invertible continuous map such that $T$ is uniquely ergodic. 
Let $\mu$ be the ergodic measure, which is regular and continuous. We will also assume that it has full support (equivalenty, that all orbits of $T$ are dense). 
For $x\in\Omega$ and $f\in\mathcal C_0(\Omega)$, the Birkhoff sum $S_{n,T}f(x)$ is now defined by $\sum_{k=0}^{n-1}f\left(T^k x\right)$. Using $\psi:\NN\to\NN$ with $\psi(n)=o(n)$ for $f\in\mathcal C_0(\Omega)$, let us define
$$\mathcal E_\psi(f)=\left\{x\in\Omega;\ \limsup_n \frac{|S_{n,T}f(x)|}{\psi(n)}=+\infty\right\}.$$

The set $\mathcal E_\psi(f)$ has already been studied by several authors. In particular, it was shown by Krengel \cite{KR78} (when $\Omega=[0,1]$) 
and later by Liardet and Voln\'y \cite{LiVol97} that, for all functions $f$ in a residual subset of $\mathcal C_0(\Omega)$, $\mu\left(\mathcal E_\psi(f)\right)=1$.
We complete this result by showing that $\mathcal E_\psi(f)$ is also residual.

\begin{theorem}\label{thm:individualresults}
Suppose that  $\psi:\NN\to\NN$ satisfies $\psi(n)=o(n)$. There exists a residual set $\mathcal R\subset\mathcal C_0(\Omega)$ such that for any $f\in \mathcal R$,
 $\mathcal E_\psi(f)$ is residual and of full $\mu$-measure in $\Omega$.
\end{theorem}

If we allow $\alpha$ to vary in our initial problem, then the natural framework now is that of topological groups. Hence,  we fix  a compact and connected metric  abelian group $(G,+)$.
By Corollary 4.4 in \cite[Chapter 4]{KuiNie},  $G$ is a monothetic group, that is possesses a dense cyclic subgroup.
Let $\mu$ be the Haar measure on $G$.
It is invariant under each translation, or group rotation $T_u(x)=x+u$. We define $G_0$ as the set of $u\in G$ such that $T_u$ is ergodic.
By well-known results of ergodic theory, $u$ belongs to $G_0$ if and only if  $\{nu;\ n\in\ZZ\}$   is dense in $G$;  in this case $T_{u}$ is uniquely ergodic, only the Haar measure is invariant with respect to $T_{u}$.   Moreover, $G_0$ is always nonempty,
it is dense and its Haar measure is equal to 1 (see Theorem 4.5 in \cite[Chapter 4]{KuiNie}).

  Contrary to what happens in Theorem \ref{thm:individualresults}, the growth of $S_{n,u}f(x)$ for a typical $(u,x)\in G^2$ is not the same from the  topological and from the probabilistic points of view.  
 For the last one, the typical growth of 
$S_{n,u}f(x)$   has order   $n^{1/2}$.

\begin{theorem}\label{thm:measuretwovariables}
\begin{enumerate}
 \item For all $\nu>1/2$ and all $f\in L^2_0(G)$, 
 $$\mu\otimes \mu\left(\left\{(u,x)\in G^2;\ \limsup_n \frac{|S_{n,u} f(x)|}{n^\nu}\geq 1\right\}\right)=0.$$
 \item There exists a residual subset $\mathcal R\subset \mathcal C_0(G)$ such that, for all $f\in\mathcal R$, 
  $$\mu\otimes \mu\left(\left\{(u,x)\in G^2;\ \limsup_n \frac{|S_{n,u} f(x)|}{n^{1/2}}=+\infty\right\}\right)=1.$$
\end{enumerate}
\end{theorem}

From a topological point of view, the typical growth  of  $S_{n,u}f(x)$   has order   $n$.
  Indeed, for $\psi:\NN\to\NN$ with $\psi(n)=o(n)$, let us  introduce 
$${\mathfrak E}_{\ppp}(f)=\left\{ (u,x)\in G^2\ ; \limsup_{n}\frac{|S_{n,u}f(x)|}{\ppp(n)}=+\oo \right\}.$$
 
\begin{theorem}\label{thm:residualtwovariables}
Suppose that  $\psi:\NN\to\NN$ satisfies $\psi(n)=o(n)$. There exists a residual set ${\mathcal R}^*\sse \mathcal C_0(G)\times G^2$ such that for any
$(f,u,x)\in {\mathcal R}^*$ we have $(u,x)\in {\mathfrak E}_{\ppp}(f)$.
\end{theorem}

We remark that, by the Kuratowski-Ulam theorem, Theorem \ref{thm:residualtwovariables} implies that there exists a residual set $\mathcal R\subset\mathcal C_0(G)$ 
such that, for every $f\in\mathcal R$, the set ${\mathfrak E}_\ppp(f)$ is residual in $G^2$. 

\medskip

The last possibility is to fix $x\in G$ and allow $u$ to vary. Without loss of generality, we may assume that $x=0$. Again, topologically speaking, the typical growth of $S_{n,u}f(0)$ is not better than $o(n)$.

\begin{corollary}\label{cor:individualresults}
Suppose that  $\psi:\NN\to\NN$ satisfies $\psi(n)=o(n)$. There exists a residual set $\mathcal R\subset\mathcal C_0(G)$ such that for any $f\in\mathcal R$, the set  $\{u\in G;\ (u,0)\in {\mathfrak E}_{\ppp}(f)\}$ is residual in $G$.
\end{corollary}

\medskip

We finally come back to irrational rotations where we would like to get more precise statements. Let us fix $\alpha\in\RR\backslash \QQ$ and set 
$$\mathcal F_\psi(f)=\left\{x\in\TT;\ \limsup_n \frac{|S_{n,\alpha}f(x)|}{\psi(n)}<+\infty\right\}.$$
When $\psi(n)=n^\nu$, $\nu\in(0,1)$, we simply denote by $\mathcal F_\nu(f)$ the set $\mathcal F_\psi(f)$.
We already know by the results mentioned before Theorem \ref{thm:individualresults} that $\lambda\left(\mathcal F_\psi(f)\right)=0$ for $f$
in a residual subset of $\mathcal C_0(\TT)$, where $\lambda$ is the Lebesgue measure on $\TT$. 
  It turns out that a much stronger result is true: generically, these sets have zero Hausdorff dimension! 

\begin{theorem}\label{thm:continuous}
 For any $\psi:\NN\to\NN$ with $\psi(n)=o(n)$, 
there exists a residual subset $\mathcal R$ of $\mathcal C_0(\TT)$ such that, for any $f\in\mathcal R$, 
$\dim_{\mathcal H}(\mathcal F_\psi(f))=0.$
\end{theorem}

We then do a similar study for H\"older functions $f\in\mathcal C_0^\xi(\TT)$, $\xi\in(0,1)$. Recall that a function $f$ belongs to $\mathcal C_0^\xi(\TT)$ if it has zero mean and if there exists a constant $C>0$ such that, for all $x,y\in\TT$, 
$$|f(x)-f(y)|\leq C|x-y|^\xi.$$
The infimum of such constants  $C$ is denoted by $\textrm{Lip}_\xi(f)$.

For a function $f\in\mathcal C_0^\xi(\TT)$, we have better bounds on $S_{n,\alpha}f(x)$ depending on $\xi$ and on  
 the 
arithmetical properties of $\alpha$. Indeed, it is known (see \cite[Chapter 2, Theorem 5.4]{KuiNie}) that $|S_{n,\alpha}f(x)|\leq n \cdot \textrm{Lip}_\xi(f) \left(D^*_n(\alpha)\right)^{\xi}$
where $D^*_n(\alpha)$ is the discrepancy of the sequence $(\alpha,2\alpha,\dots,n\alpha)$ defined by 
$$|D^*_n(\alpha)|=\sup_{I\subset\TT} \left|\frac{\textrm{card}\{1\leq i\leq n;\ i\alpha\in I\}}{n}-|I|\right|.$$
For instance, if $\alpha$ has type 1 (for example, if $\alpha$ is an irrational algebraic number), using the well-known estimates 
of the discrepancy, we get that $|S_{n,\alpha}f(x)|=O(n^{1-\xi+\veps})$ for all $\veps>0$. In other words, for all $\nu>1-\xi$, $\mathcal F_\nu(f) =\TT$. 
We investigate the case $\nu\leq 1-\xi$ and we show that   the Hausdorff dimension of $\mathcal F_\nu (f) $ cannot always be large.

\begin{theorem}\label{thm:holder}Let  $\xi\in(0,1)$.
There exists $f\in\mathcal C^{\xi}_0(\TT)$ such that, for all $\nu\in(0,1-\xi)$, 
$$\dim_{\mathcal H}(\mathcal F_\nu(f))\leq\sqrt{\frac{\xi}{1-\nu}}.$$
\end{theorem}

This theorem is in stark contrast with Theorem 4.1 in \cite{FanSchme03}. In this last paper, a similar study of fast Birkhoff averages of subshifts is done.
In this case, the sets which correspond to $\mathcal F_\nu(f)$ always have  maximal dimension.

\section{Useful lemmas}
 
In this section, we provide lemmas which will be used several times for the proof of our main theorems.
The first one allows to approximate step functions by continuous functions. 
In the statement of the theorem we use the standard notation $\mathbf 1_{B}(x)$
for the function which equals $1$ if $x\in B$ and equals 0 if not.

\begin{lemma}\label{lem:stepapproximation}
   Let $\Omega$ be a compact metric space, let $\mu$ be a continuous Borel probability measure on $\Omega$.  \ 
   Let $g$ be a step function such that $\int_\Omega g(x)d\mu(x)=0$ and $\delta>0$. Then there exists $f\in\mathcal C_0(\Omega)$ such that $\|f\|_\infty\leq 2\Vert g\Vert_\infty$ and 
 $f=g$ except on a set of measure at most $\delta$.
\end{lemma}
\begin{proof}
 Let $\veps>0$ be very small and $\{a_1,\dots,a_n\}$ be the finite set $g(\Omega)$. We can write $g=\sum_{i=1}^na_i \mathbf 1_{A_i}$ where $A_i=\{x\in\Omega\ ;\ g(x)=a_i\}$. Since the measure $\mu$ is regular, we can find compact sets $K_1,\dots,K_n$ and open sets
 $U_1,\dots,U_n$ such that 
 \[ K_i\subset A_i\subset U_i\]
 \[ \mu(U_i)-\veps \leq \mu(A_i)\leq \mu(K_i)+\veps. \]
 By Urysohn's lemma, one may find functions $\varphi_i\in\mathcal C(\Omega)$ such that
 \[ 0\leq\varphi_i\leq 1,\quad \varphi_i=1\textrm{ on }K_i,\quad \varphi_i=0\textrm{ outside }U_i.\]
We then set $h=\sum_{i=1}^n a_i\varphi_i$. It is clear that 
\[ \mu\left(\left\{a_i\mathbf 1_{A_i}\neq a_i\varphi_i\right\}\right)\leq \mu\left(U_i\backslash K_i\right). \]
Therefore,
\[
 \mu\left(\left\{h\neq g\right\}\right)\leq \sum_{i=1}^n \mu\left(U_i\backslash K_i\right)\leq 2n\veps.
\]

If $k=\max\left(-\Vert g\Vert_\infty,\min\left(h,\Vert g\Vert_\infty\right)\right)$, we now have $\Vert k\Vert_\infty\leq \Vert g\Vert_\infty$ and
\[
 \mu\left(\left\{k\neq g\right\}\right)\leq  \mu\left(\left\{h\neq g\right\}\right)\leq 2n\veps.
 \]

The function  $k$ is continuous but is not necessarily in $\mathcal C_0(\Omega)$. Nevertheless, we observe that
\[
\left|\int_{\Omega}  k(x)d\mu(x)\right|=\left|\int_{\Omega} \left( k(x)-g(x)\right)d\mu(x)\right| \leq \Vert k-g\Vert_\infty \mu\left(\left\{k\neq g\right\}\right)\leq 4n\veps\Vert g\Vert_\infty
\]
and we can modify $k$ to obtain a zero mean.  Let $a\in\Omega$ and   $r>0$ be such that $ 0<\mu\left(B(a,r)\right)\leq \mu\left(B(a,2r)\right)<\delta/2$ and let $\varphi_0\in\mathcal C(\Omega)$ with $\varphi_0=1$ on the closed ball  $\bar B(a,r)$, $\varphi_0=0$
outside $B(a,2r)$ and $0\leq\varphi_{0}\leq 1$. We set 
\[ f=k-\frac{\int_\Omega kd\mu}{\int_\Omega \varphi_0 d\mu}\varphi_0.\]
Then $f\in\mathcal C_0(\Omega)$, $f=g$ except on a set of measure at most $2n\veps+\delta/2$ and
\[\|f\|_\infty\leq \Vert g\Vert_\infty+\frac{\left|\int_\Omega kd\mu\right|}{\int_\Omega \varphi_0 d\mu}\leq \Vert g\Vert_\infty+\frac{4n\veps\Vert g\Vert_\infty}{\mu\left(B(a,r)\right)}.\]
Choosing $\veps>0$ sufficiently small then gives the result.
\end{proof}

Our second lemma is a way to construct continuous functions  in $\mathcal C_0(\Omega)$ with large Birkhoff sums on large subsets. We give it in our general context of a uniquely ergodic transformation
$T$  on an infinite compact metric space $\Omega$  with  non-atomic ergodic measure $\mu$. As usual, $\psi:\NN\to\NN$ satisfies $\psi(n)=o(n)$.
We denote by $E^{c}$ the complement of the set $E$.

\begin{lemma}\label{lem:bigpartialsums}
Let $J,M\in\NN$, $C>0$, $  \veps  >0$. Then there exist $f\in\mathcal C_0(\Omega)$, $m\geq M$ and a compact set $E\subset \Omega$
 such that $\|f\|_\infty\leq  \veps  $, $\mu(E)>1-  \veps  $ and
 \begin{align*}
  &\forall x\in E,\ \forall j\in\{1,\dots,J\},\quad \left|S_{m,T^j}f(x)\right|\geq C\psi(m).
 \end{align*}
\end{lemma}

\begin{proof}
Set $ \overline{\veps}  =\veps/3.$ 
  We begin by  fixing {$m\in\NN$, any integer greater than $M$,} and such that $m  \overline{\veps} \geq C\psi(m)$.
  Let $n\gg m$ to be fixed  later. We then consider a Rokhlin tower  associated to $T$, $2n$ and $ \overline{\veps}$ (see for instance \cite{eisner2015operator}). Namely, we consider
 $A\subset \Omega$ such that the sets $T^k (A)$, $0\leq k\leq 2n-1$, are pairwise disjoint and $\mu\left(\bigcup_{k=0}^{2n-1}T^k (A)\right)>1- \overline{\veps}$. 
We then consider a function $g$ equal to $ \overline{\veps} $ on $\bigcup_{k=0}^{n-1}T^k(A)$, equal to $-  \overline{\veps} $ on 
 $\bigcup_{k=n}^{2n-1}T^k(A)$ and equal to zero elsewhere. 
 
 We set 
 \[F= \left(\bigcup_{k=0}^{n-1-mJ}T^k(A)\right)\cup\left(\bigcup_{k=n}^{2n-1-mJ}T^k(A)\right):=F_1\cup F_2.\]
 Then, for any $x\in F_1$, for any $ \ell \leq m-1$, for any $j\in\{1,\dots,J\}$, 
\[T^{\ell j}(x)\in\bigcup_{ k=0}^{n-1}T^{k} (A).\]
It follows that 
$S_{m,T^j}g(x) =m \overline{\veps}$.
In the same way, for any $x\in F_2$, for any $j\in\{1,\dots,J\}$, $S_{m,T^j}g(x) =-m \overline{\veps}$.

Finally, for any $x\in F$, for any $j\in\{1,\dots,J\}$,
\[\left\vert S_{m,T^j}g(x) \right\vert=m \overline{\veps}\geq C\psi(m).\]
Moreover, 
\[\mu(F)=2\left( n-mJ\right)\mu(A)\geq2\left(n-mJ\right)\cdot \frac{1-\overline{\veps}}{2n}\geq1-2\overline{\veps}\]
provided $n$ is large enough.
 
 Thanks to Lemma \ref{lem:stepapproximation}, we approximate $g$ by a continuous function $f\in\mathcal C_0(\Omega)$ with $\|f\|_\infty\leq 2  \overline{\veps} $ and $f=g$ except
 on a set $\mathcal N$ of measure $\eta>0$, with $mJ\eta<  \overline{\veps} $. Fix $j\in\{1,\dots,J\}$. 
 Then $S_{m,T^j} f(x)=S_{m,T^j}g(x)$ except if $x\in\bigcup_{k=0}^{m-1}T^{-kj}(\mathcal N)$. Let  $\mathcal N'=\bigcup_{k=0}^{m-1}\bigcup_{j=1}^JT^{-kj}(\mathcal N)$.
 Then $\mu(\mathcal N')\leq mJ\eta< \overline{\veps}$. Moreover, $\left|S_{m,T^j}f(x)\right|\geq C\psi(m)$ for all $j\in\{1,\dots,J\}$ and all $x\in F\cap\mathcal N'^c=:E_0$. 
 Clearly, $\mu(E_{0}) > 1-3 \overline{\veps} =1-\veps$.  We conclude by taking for $E$ the closure of $E_0$.
\end{proof}

\section{Fast and slow points of Birkhoff sums - I}

In this section, we prove Theorems \ref{thm:individualresults} and \ref{thm:residualtwovariables}. Their proofs share many similarities and depend heavily on Lemma \ref{lem:bigpartialsums} applied in suitable situations. 
We will also need that if $T$ is a uniquely ergodic transformation on $\Omega$, then the set of $\mathcal C_0(\Omega)$-coboundaries  for $T$, namely the set of functions $g-g\circ T$ for some $g\in\mathcal C_0(\Omega)$, is dense in 
$\mathcal C_0(\Omega)$ 
(see for instance \cite[Lemma 1]{LiVol97}).
It is convenient to work with a coboundary since its Birkhoff sums are uniformly bounded.

\begin{proof}[Proof of Theorem \ref{thm:individualresults}]
Let $(h_l)$ be a dense sequence of coboundaries in $\mathcal C_0(\Omega)$ and let $C_l>0$ be such that 
 $\sup_n \left\|S_{n,T}h_l\right\|_\infty\leq C_l.$
 Let $f_l$, $E_l$ and $m_l$ be given by Lemma \ref{lem:bigpartialsums} for $C=l+C_l+1$, $M=l$, $J=1$, $\veps=1/l$.  
 We set $g_l=h_l+f_l$ and we observe that, for $x\in E_l$, 
 $$\left|S_{m_l,T}g_l(x)\right|\geq (l+C_l+1)\psi(m_l)-C_l\geq (l+1)\psi(m_l).$$
 Since $E_l$ is compact and $g_l$ is continuous, we can choose $\delta_l>0$ and an open set $F_l\subset \Omega$ containing $E_l$ such that, 
 for any $f\in B(g_l,\delta_l)$, for any $x\in F_l$, 
 \begin{equation}\label{eq:individual1}
  \left|S_{m_l,T}f(x)\right|\geq l\psi(m_l).
 \end{equation}
 Let $\mathcal R=\bigcap_{L\geq 1}\bigcup_{l\geq L}B(g_l,\delta_l)$ which is a residual set in $\mathcal C_0(\Omega)$ and pick $f\in\mathcal R$. There exists an increasing sequence $(l_k)$
 going to $+\infty$ such that $f\in B(g_{l_k},\delta_{l_k})$ for all $k$. We set $F =\limsup F_{l_k}=\bigcap_{K\geq 1}\bigcup_{k\geq K}F_{l_k}$. 
 Since $\mu(F_{l_{k}})\geq \mu(E_{l_{k}})\geq 1-\frac{1}{l_{k}}$  the set $F$ has full measure. Moreover, since $\mu$ has full support and $\mu\left(\bigcup_{k\geq K}F_{l_k}\right)=1$ for all $K$, $F$ is also residual in $\Omega$.
Finally   if $x$ belongs to $F$, then \eqref{eq:individual1} is true for infinitely many $l$, which shows Theorem \ref{thm:individualresults}.
\end{proof}

 In the next proof $\Omega$ is replaced by the compact connected metric abelian group $G$ and we consider uniquely ergodic translations $T_{v}$.  We recall that for these translations, all non-constant characters $\gamma$ are $\mathcal C_0$-coboundaries:
 they can be written as  $\gamma=\gamma_0\circ T_v-\gamma_0$, where $\gamma_0=\frac1{\gamma(v)-1}\gamma$.
 
\begin{proof}[Proof of Theorem \ref{thm:residualtwovariables}]
Since $G$ is compact we can choose a sequence $(h_l)$ 
 of trigonometric polynomials which is dense in $\mathcal C_0(G)$ (see \cite[Section 1.5.2]{Rudin}). Let $v\in G_0$,
 that is $T_{v}$ is ergodic.
 Since $h_l$ is a $\mathcal C_0$-coboundary
for all $T_{jv}$, $j=1,\dots,l$, there
 exists $C_l>0$ such that 
 $$\sup_n \sup_{j\in\{1,\dots,l\}}\left\|S_{n,j v}h_l\right\|_\infty\leq C_l.$$
 Let $f_l$, $E_l$ and $m_l$ be given by Lemma \ref{lem:bigpartialsums} for $T=T_v$, $C=C_l+l+1$, $M=l$, $J=l$, $\veps=1/l$. 
 Set $g_l=h_l+f_l$ and observe that, for $x\in E_l$, $j\in\{1,\dots,l\}$, 
 $$\left|S_{m_l,j v}g_l(x)\right|\geq (l+C_l+1)\psi(m_l)-C_l\geq (l+1)\psi(m_l).$$
 Since $\{j v;\ j=1,\dots,l\}\times E_l$ is compact in $G\times G$ and $g_l$ is continuous, we can choose $\delta_l>0$ and an open set $H_l\subset G\times G$ such that
 $\{j v;\ j=1,\dots,l\}\times E_l\subset H_l$ and, for any $(f,u,x)\in B(g_l,\delta_l)\times H_l$, 
 \begin{equation}\label{eq:residualtwovariables}
  \left|S_{m_l,u}f(x)\right|>l\psi(m_l).
 \end{equation}
 We now observe that $\bigcup_{l\geq L}\{g_l\}\times \left\{T^j v;\ j=1,\dots,l\right\}\times E_l$ is dense in $\mathcal C_0(G)\times G\times G$ for any $L\geq 1$. Hence, 
 ${\mathcal R}^*= \bigcap_{L\geq 1}\bigcup_{l\geq L}B(g_{l},\ddd_l)\times {H}_l$ is a residual subset of $\mathcal C_0(G)\times G^2$ 
 and any $(f,u,x)\in\mathcal R^*$ satisfies that $(u,x)$ belongs to ${\mathfrak E}_{\ppp}(f)$ since \eqref{eq:residualtwovariables} is true for infinitely many integers $l$.
\end{proof}

\begin{proof}[Proof of Corollary \ref{cor:individualresults}]
This corollary follows easily from Theorem \ref{thm:residualtwovariables} and from the Kura\-towski-Ulam theorem. Indeed, we know that there exist a residual set $R\subset\mathcal C_0(G)$ and $x\in G$ such that, for all $f\in\mathcal R$, $\{u\in G;\ (u,x)\in{\mathfrak E}_\ppp(f)\}$ is residual. Now, setting $\mathcal R'=\{f(\cdot-x);\ f\in\mathcal R\}$, for any $f\in\mathcal R'$, $\{u\in G;\ (u,  0)\in{\mathfrak E}_\ppp(f)\}$ is residual.
\end{proof}

\section{Fast and slow points of Birkhoff sums - II}\label{sec:as}

We turn to the   proof of Theorem \ref{thm:measuretwovariables}.
Its first part heavily depends on the following Menshov-Rademacher inequality (see for instance \cite[Chapter 4]{doob}). 
\begin{lemma}\label{lem:menshov}
Let $X_1,\dots,X_N$ be a sequence of orthonormal random variables and $c_1,\dots,c_N$ be a sequence of real numbers. Then 
$$\mathbb E\left(\max_{1\leq n\leq N}\left(\sum_{j=1}^n c_j X_j\right)^2\right)\leq \log^2_{2}(4N)\sum_{n=1}^N c_n^2.$$
\end{lemma}
\begin{proof}[Proof of Theorem \ref{thm:measuretwovariables} part {(i)}]
Recall that $\int_G f(x)d\mu(x)=0$. Without loss of generality, we suppose $\|f\|_2=1$ and we consider $X_k(u,x)=f(x+ku)$ as a random variable 
on the probability space $(G^2,\mu\otimes\mu)$. Next we show that $(X_k)_{k\geq 1}$
is an orthonormal sequence. Indeed, let $\sum_{\gamma\in\hat G}\hat f(\gamma)\gamma$ be the Fourier
expansion of $f$. Then, for $k,j\geq 1$, 
\[
\int_{G^2}X_k\overline{X_j}d\mu\otimes d\mu=\sum_{\gamma,\gamma'\in\hat G}\hat f(\gamma)\overline{\hat{f}(\gamma')}\int_{G}\gamma(x)\overline{\gamma'(x)}d\mu(x)\int_G \gamma(ku)\overline{\gamma'}(ju)d\mu(u).
\]
Now, $\int_G \gamma(x)\overline{\gamma'(x)}d\mu(x)$ is zero provided $\gamma\neq\gamma'$ and is equal to 1 otherwise. Moreover, let us fix $\gamma\in\hat G$ and set $\gamma_k(u)=\gamma(ku)$, 
$\gamma_j(u)=\gamma(ju)$. Then $\int_G \gamma_k\overline{\gamma_j}d\mu=0$ except if $\gamma_k=\gamma_j$, namely except if $\gamma^{k-j}=1$. If $k\neq j$,  using that $\hat G$ is torsion-free since 
$G$ is compact and connected, this can only happen if $\gamma=1$. Therefore, we have shown that
\[\int_{G^2}X_k\overline{X_j}d\mu\otimes d\mu =\left\{\begin{array}{ll}
           \sum_{\gamma}|\hat f(\gamma)|^2=1&\textrm{ if }k=j\\
           |\hat f(1)|^2=0&\textrm{ otherwise.}
          \end{array}\right.
\]

Applying Lemma \ref{lem:menshov} with $c_j=1$ yields
\begin{equation}
\int_{G^2}\max_{1\leq n\leq N}|S_{n,u}f(x)|^2d\mu(u)\otimes d\mu(x) \leq \log^2_{2}(4N)N.
\label{eq:asnottoofast1}
\end{equation}
Let $\nu>1/2$ and for $k\geq 1$,
\[ 
E_k=\left\{(u,x)\in G^2;\ \exists n\in \{2^k,\dots,2^{k+1}-1\},\  |S_{n,u}f(x)|\geq n^\nu\right\}.
\]
Using Markov's inequality and \eqref{eq:asnottoofast1}, we get 
\[
\mu\otimes\mu(E_k)\leq \mu\otimes\mu\left( \max_{1\leq n\leq 2^{k+1}}|S_{n,u}f(x)|\geq 2^{\nu k}\right)\leq 
\frac1{2^{2k\nu}}\log_2^2{ \left( 4 \cdot  2^{k+1}\right) \cdot  2^{k+1}}\leq Ck^{2}2^{k(1-2\nu)}.
\]

Since $\sum_k \mu\otimes\mu(E_k) <\oo  $, the Borel-Cantelli lemma implies that $\mu\otimes\mu(\limsup_k E_k)=0$ and the conclusion follows.
\end{proof}
\begin{remark}
In fact, the same proof shows that, for any $\veps>0$, 
\[
\mu\otimes\mu\left(\left\{(u,x)\in G^2;\ \limsup_n \frac{|S_{n,u}f(x)|}{n^{\frac 12}\log^{\frac{3}2+\veps}(n)}\geq 1\right\}\right)=0.
\]
\end{remark}


To prove the second part of Theorem \ref{thm:measuretwovariables}, we shall use both a Baire category and a probabilistic argument. The probabilistic part is based on the the following lemma,
which is a consequence of the proof of the law of the iterated logarithm   done in \cite{Bil95}   (the important point here is that we need a choice of $N$ which does not depend
on the particular choice of the sequence).

We recall that a random variable $X:(\Omega,\mathcal A,P)\to\RR$ has a Rademacher distribution if $P(X=1)=P(X=-1)=1/2$.

\begin{lemma}\label{lem:LIL}
 Let $\veps>0$ and $M\in\NN$. There exists $N\geq M$ such that, for any sequence $(Y_k)$ of independent Rademacher variables defined on the same probability space
 $(\Omega,\mathcal A,P)$, 
 $$P\left(\sup_{M\leq n\leq N}\frac{|\sum_{k=1}^n Y_k(\omega)|}{\sqrt{n\log\log n}}>\frac 12\right)>1-\veps.$$
\end{lemma}

The following lemma is the key point of our proof.

\begin{lemma}\label{lem:key}
Let $\veps\in(0,1)$, $C>0$ and $M\in\NN$. There exist $f\in\mathcal C_0(G)$, $N>M$ and $F\subset G^2$ with $\|f\|_\infty\leq \veps$,
$\mu\otimes\mu(F)>1-\veps$ and 
$$(u,x)\in F\implies \sup_{M\leq n\leq N}\frac{|S_{n,u}f(x)|}{n^{1/2}}\geq C.$$
\end{lemma}

\begin{proof}
Without loss of generality, we may assume that $\sqrt{\log\log M}> 2C/\veps$. Lemma \ref{lem:LIL} gives us a value of $N$ associated
to $\veps$ and $M$. We then consider a sequence $(X_k)$ of independent Rademacher variables defined on the same probability space $(\Omega,\mathcal A,P)$.
We select a neighbourhood $\mathcal O$ of $0\in G$ so that, setting
\[
E_{{\mathcal O}}=\left\{ u\in G;\ 
(j'-j)u\notin 2\mathcal O\textrm{ for all } 0\leq j,j' \leq N,\ j\not=j' \right\},
\]

we have $\mu(E_\mathcal O)>1-\veps$. This is possible since, denoting by $(\mathcal O_l)$ a basis of neighbourhoods of $0$ in $G$, we have
$$G_0\subset\left\{u\in G;\ ku\not=0 \textrm{ for all }k\in\ZZ\backslash\{0\}\right\}\subset\bigcup_l E_{\mathcal O_l}.$$
By compactness of $G$, $G$ is contained in a finite union $(x_1+\mathcal O)\cup\cdots\cup (x_K+\mathcal O)$. We set $A_1=x_1+\mathcal O$ and, 
for $2\leq k\leq K$, $A_k=(x_k+\mathcal O)\backslash(A_1\cup\dots\cup A_{k-1})$. The sets $A_1,\dots,A_k$  provide a  Borelian partition of $G$. 

We then split each $A_k$ into a disjoint sum $A_k=B_k\cup B'_k$ with $\mu(B_k)=\mu(B'_k)=\mu(A_k)/2$. 
For $1\leq k \leq {K}$ 
define $\varphi_k$ by $\varphi_k=\left(\mathbf 1_{B_k}-\mathbf 1_{B'_k}\right)$. We finally put
$$g(x,\ooo)=\sum_{k=1}^{{K}}\veps X_k(\ooo)\fff_{k}(x)$$
so that
$$S_{n,u}g(x,\omega)=\veps\sum_{j=0}^{n-1}\sum_{k=1}^K X_k(\omega)\varphi_{k}(x+ju).$$

Let us fix $u\in E_\mathcal O$. For all $x\in G$ and all $j\in\{0,\dots,N-1\}$,
there exists exactly one integer $k\in\{1,\ldots,K\}$, that we will denote  by  $k(j,u,x)$, such that $\varphi_k(x+ju) \not =0$.
Hence, 
for $(u,x)\in E_\mathcal O\times G$ and $n\leq N$,
$$S_{n,u}g(x,\omega)=\veps\sum_{j=0}^{n-1}X_{k(j,u,x)} (\ooo) \varphi_{k(j,u,x)}(x+ju).$$ 
Moreover, for $j\neq j'$, the integers $k(j,u,x)$ and $k(j',u,x)$ are different:  otherwise,    $(j-j')u$ would  belong  to $2\mathcal O$. 

Applying Lemma \ref{lem:LIL} to the sequence $\big(X_{k(j,u,x)}\varphi_{k(j,u,x)}(x+ju)\big)_{0\leq j\leq N-1}$ which is a sequence of independent Rademacher variables, 
we get the existence of $\Omega_{u,x}\subset\Omega$ such that $P(\Omega_{u,x})>1-\veps$ and 
$$(u,x,\omega)\in E_{\mathcal O}\times G \times \Omega_{u,x}\implies
\sup_{M\leq n\leq N}\frac{|S_{n,u}g(x,\omega)|}{\sqrt{n\log\log n}}\geq\frac\veps2.$$
Hence
$$(u,x,\omega)\in E_{\mathcal O}\times G\times \Omega_{u,x}\implies
\sup_{M\leq n\leq N}\frac{|S_{n,u}g(x,\omega)|}{\sqrt{n}}\geq \frac\veps2{ \sqrt{\log\log M}}>C.$$

Keeping in mind that $\mu(E_{\mathcal O})>1-\veps$ holds as well, by Fubini's theorem  we can select and fix $\ooo\in\OOO$
 such that 
 \begin{equation}\label{*efk}
\mu\otimes\mu\Big (\Big \{ (u,x)\in G^{2};\sup_{M\leq n\leq N}\frac{|S_{n,u}g(x,\ooo)|}{\sqrt{n}}
>C 
\Big \}\Big )>(1-\veps)^2>1- 2\veps.
\end{equation}
Given $\delta>0$, according to Lemma \ref{lem:stepapproximation}, the function $g=g(\cdot,\ooo)$ can be approximated by a continuous function $f\in\mathcal C_0(G)$ such that $\Vert f\Vert_\infty\leq2\veps$ and which coincides with $g$ except in a set of measure less than $\delta/N$. It follows that for every $u\in G$ and for any $n\in\{M,\ldots,N\}$, $S_{n,u}f(x)=S_{n,u}g(x)$ except in a set of measure less than $\delta$. Finally, if $\delta$ is sufficiently small, inequality \eqref{*efk} is still satisfied if we replace $g$ by $f$. 
\end{proof}

\begin{proof}[Proof of Theorem \ref{thm:measuretwovariables}, part {(ii)}]
 Let $(h_l)$ be a sequence of trigonometric polynomials dense in $\mathcal C_0(G)$. For all $l\geq 1$ and all $u\in G_0$, 
 since $h_l$ is a $\mathcal C_0$-coboundary for $T_u$,
we know that $\sup_n \left\|S_{n,u}h_l\right\|_\infty<+\infty$. We then find 
 $G_l\subset G_0$ with $\mu(G_l)>1-1/l$ and $C_l>0$ such that, for all $u\in G_l$, $\sup_n \left\|S_{n,u}h_l\right\|_\infty\leq C_l$. 
 We apply Lemma \ref{lem:key} with $\veps=1/l$, $C=l+C_l+1$ and $M_l=l$. We get a function $f_l\in\mathcal C_0(G)$, an integer $N_l\geq M_l$
 and a set $F_l\subset G^2$. We define $g_l=h_l+f_l$ and $E_l=F_l\cap(G_l\times G)$ so that $\mu\otimes\mu(E_l)\geq 1-2/l$. 
 The way we constructed all these objects
 ensures that, for any $(u,x)\in  E_l$, 
 $$\sup_{M_l\leq n\leq N_l}\frac{|S_{n,u}g_l(x)|}{n^{1/2}}\geq l+1.$$
 This yields the existence of a $\delta_l>0$ such that, 
 for any $f\in B(g_l,\delta_l)$ and any $(u,x)\in E_l$,
 $$\sup_{M_l\leq n\leq N_l}\frac{|S_{n,u}f(x)|}{n^{1/2}}\geq l.$$
 We finally consider the residual set $\mathcal R=\bigcap_{L\geq 1}\bigcup_{l\geq L}B(g_l,\delta_l)$ and we pick $f\in\mathcal R$. There
 exists an increasing sequence $(l_k)$ such that $f\in B(g_{l_k},\delta_{l_k})$. Let $E=\limsup_k E_{l_k}$ which has full measure and pick $(u,x)\in E$. There
 exists a subsequence $(l'_k)$ of $(l_k)$ such that $(u,x)\in E_{l'_k}$ for all $k$. We then have
 $$\sup_{M_{l'_k}\leq n\leq N_{l'_k}}\frac{|S_{n,u}f(x)|}{n^{1/2}}\geq l'_k$$
 which allows us to conclude.
\end{proof}

\begin{remark}
The proof gives slightly more than announced: there exists a residual set $\mathcal R\subset \mathcal C_0(G)$ such that, for all $\veps\in(0,1/2)$ and all $f\in\mathcal R$, 
  $$\mu\otimes \mu\left(\left\{(u,x)\in G^2;\ \limsup_n \frac{|S_{n,u} f(x)|}{n^{1/2}(\log\log n)^{\frac 12-\veps}}=+\infty\right\}\right)=1.$$
\end{remark}

\section{Fast and slow points for irrational rotations on the circle}
Throughout this section, we fix $\alpha\in\mathbb R\backslash\mathbb Q$.
\subsection{A partition of $\TT$}\label{sec:partition}

To get an estimate of the Hausdorff dimension of $\mathcal F_\psi(f)$, which is more precise  than the result already obtained on its measure, 
we will need a refinement of Rokhlin towers specific to irrational rotations.
We shall use the following system of partitions of $\TT$ associated to the irrational rotation $R_\alpha$, as it is described for instance in \cite[Lecture 9, Theorem 1]{Sinaitopics}.
Let $(p_n/q_n)$ be the $n$-th convergent of $\alpha$ in its continued fraction expansion. Define
$$\Delta_0^{(n)}=\left\{
\begin{array}{ll}
 [0,\{q_n \alpha\})&\textrm{ if $n$ is even}\\
 \left[ \{q_n \alpha\},1\right)&\textrm{ if $n$ is odd.}
\end{array}
\right.$$
Denote also $\Delta_j^{(n)}=R_\alpha^j\big(\Delta_0^{(n)}\big)$. For any $n\geq 1$, the intervals $\Delta_j^{(n)}$, $0\leq j<q_{n+1}$ and $\Delta_j^{(n+1)}$, $0\leq j<q_n$, are
pairwise disjoint and their union is the whole $\TT$. We shall denote by $d_n$ the length of $\Delta_0^{(n)}$. It is well known that 
$$\frac1{2q_{n+1}}\leq d_n\leq\frac1{q_{n+1}}.$$

\subsection{Continuous functions}

The main step towards the proof of Theorem \ref{thm:continuous} is the following lemma   which improves partly Lemma \ref{lem:bigpartialsums}. 

\begin{lemma}\label{lem:continuous}
Let $M\in\mathbb N$, $C>0$, $s\in(0,1)$, $\delta>0$ and  $\veps>0$.
Then there exist  $f\in \mathcal C_0(\TT)$ with $\|f\|_\infty\leq \veps$, a compact set $E\subset\TT$, and an integer $m\geq M$ such that
\begin{equation}\label{eq:continuous1}
 \forall x\in E,\ |S_{m,\alpha}f(x)|\geq C\psi(m);
\end{equation}
\begin{equation}\label{eq:continuous2}
\mathcal H^s_\delta(E^c)<\veps .
\end{equation}
\end{lemma}
\begin{proof}
Let $m\geq M$ be such that $m\veps>C\psi(m)$. Let also $n$ be a large integer and consider the partition of $\TT$ described in Section \ref{sec:partition}: 
 $$\TT=\bigcup_{0\leq j<q_{n+1}}\Delta_j^{(n)}\cup \bigcup_{0\leq j<q_n}\Delta_j^{(n+1)},$$
where the convergents of $\aaa$ are $p_{n}/q_{n}$.
Since it will be easier to deal with even numbers we put $\widetilde q_{n}=2\lf q_{n}/2 \rf$, $n\in\NN$ which is the  greatest even integer less than $q_n$. Hence $\widetilde q_n$ and $\widetilde q_{n+1}$ are even. We define a continuous function $f$ with zero mean such that

 \begin{itemize}
 \item on $\Delta_j^{(n)}$, $0\leq j< \frac{\widetilde q_{n+1}}2$ and on 
 $\Delta_j^{(n+1)}$, $0\leq j< \frac{\widetilde q_n}2$, $f=\veps$ except on two very small intervals of size $\eta>0$
 where $f$ is affine to ensure that $f$ vanishes at the boundary of $\Delta_j^{(n)}$ and $\Delta_j^{(n+1)}$.
 \item on $\Delta_j^{(n)}$, $\frac{\widetilde q_{n+1}}2\leq j<\widetilde q_{n+1}$ and on $\Delta_j^{(n+1)}$, $\frac{\widetilde q_n}2\leq j<\widetilde q_n$, $f=-\veps$ except on two very small intervals of size $\eta>0$ where $f$ is affine to ensure that $f$ vanishes at the boundary of $\Delta_j^{(n)}$ and $\Delta_j^{(n+1)}$.
 \item if $x\not \in \bigcup_{0\leq j<\widetilde q_{n+1}}\Delta_j^{(n)}\cup \bigcup_{0\leq j<\widetilde q_n}\Delta_j^{(n+1)}$
we set $f(x)=0$.
 \end{itemize}

 We set $\Gamma_j^{(n)}$ (resp. $\Gamma_j^{(n+1)}$) the (largest) subinterval of $\Delta_j^{(n)}$ (resp. $\Delta_j^{(n+1)}$) such that $|f|=\veps$ and we let 
 $$E=\bigcup_{0\leq j< \frac{\widetilde q_{n+1}}2-m}\Gamma_j^{(n)}\cup\bigcup_{\frac{\widetilde q_{n+1}}2\leq j<\widetilde q_{n+1}-m}\Gamma_j^{(n)}\cup\bigcup_{0\leq j< \frac{\widetilde q_{n}}2-m}\Gamma_j^{(n+1)}\cup\bigcup_{\frac{\widetilde q_{n}}2\leq j<\widetilde q_n-m}\Gamma_j^{(n+1)}.$$
 If $x$ belongs to $E$, then $f(x+j \alpha)=f(x)$ for all $j=0,\dots,m-1$ and $|f(x)|=\veps$. Therefore,
 we have $|S_{m,\alpha}f(x)|=m\veps>C\psi(m)$. On the other hand, $E^c$ is the union of at most 
\begin{itemize}
\item $(2m+2)$ intervals of size ${d_n}$;
\item $(2m+2)$ intervals of size $d_{n+1}$;
\item $2(\widetilde q_{n+1}+\widetilde q_n)$ intervals of size $\eta$.
\end{itemize}
Hence we have
$$\mathcal H^s_\delta(E^c)\leq (2m+2) d_n^s+(2m+2) d_{n+1}^s +2(\widetilde q_{n+1}+\widetilde q_n) \eta^s<\veps$$
if we choose $n$ sufficiently large and then $\eta$ sufficiently small. 
\end{proof}
\begin{proof}[Proof of Theorem \ref{thm:continuous}] We mimic the proof of Theorem \ref{thm:individualresults}.  Recall that
$$\mathcal F_\psi(f)=\left\{x\in\TT;\ \limsup_n \frac{|S_{n,\alpha}f(x)|}{\psi(n)}<+\infty\right\}.$$
Let $(h_l)$ be a sequence of coboundaries which is dense in $\mathcal C_0(\TT)$.
Then for any $l\geq 1$, there exists $C_l>0$ such that $\sup_n\|S_{n,\alpha}h_l\|_\infty\leq C_l$.
Let $f_l$, $E_l$ and $m_l$ be given by Lemma \ref{lem:continuous} for $C=l+C_l+1$, $M=l$ and $\veps=s=\delta={1/l}$. We set $g_l=h_l+f_l$ and observe that, for $x\in E_l$, 
$$|S_{m_l,\alpha}g_l(x)|\geq (l+C_l+1)\psi(m_l)-C_l\geq (l+1)\psi(m_l).$$
There exists $\delta_l>0$ such that, for any $f\in B(g_l,\delta_l)$ and any $x\in E_l$, 
$$|S_{m_l,\alpha}f(x)|\geq l\psi(m_l).$$
Since the sequence $(g_l)$ is dense in $\mathcal C_0(\TT)$, $\mathcal R=\bigcap_{L\geq 1}\bigcup_{l\geq L}B(g_l,\delta_l)$ is a residual subset of $\mathcal C_0(\TT)$. Pick $f\in\mathcal R$.
There exists an increasing sequence $(l_k)$ such that $f\in B(g_{l_k},\delta_{l_k})$. We set $E=\limsup E_{l_k}$ and observe that, 
for any $x\in E$, 
$$\limsup_n \frac{|S_{n,\alpha}f(x)|}{\psi(n)}=+\infty.$$
Moreover, $E^c=\bigcup_{K\geq 1}\bigcap_{k\geq K}E^c_{l_k}$. For any $s\in (0,1)$, the properties of the sets $E_l$ ensure that $\mathcal H^s\left(\bigcap_{k\geq K}E^c_{l_k}\right)=0$. Since $\mathcal F_\psi\subset E^c$, we conclude that $\dim_{\mathcal H}(\mathcal F_\psi)\leq s$ and therefore $\dim_{\mathcal H}(\mathcal F_\psi)=0$.
\end{proof}

\subsection{H\"older functions}

  We now modify the previous construction to adapt it to H\"older continuous functions. 

\begin{lemma}\label{lem:holder}
 Let $M\in\NN$, $\nu\in(0,1)$, $\xi\in(0,1)$ with $\nu+\xi<1$, $A>0$, $\sqrt{\frac{\xi}{1-\nu}}< s\leq 1$, $\delta>0$, $\veps>0$.
 There exist a continuous function $f  \in\mathcal C_0(\TT)$ with $\| f \|_\infty\leq 1$, $\mathrm{Lip}_\xi( f )\leq 1$, an integer  $N\ge M$, and  a compact set $E\subset\TT$ such that
\begin{equation}\label{eq:holder1}
 \forall x\in E,\quad \exists m\in\{M,\dots,N\},\quad |S_{m,\alpha} f (x)|\geq A m ^\nu,
\end{equation}
\begin{equation}\label{eq:holder2}
\mathcal H^s_\delta(E^c)<{ \veps.}
\end{equation}
\end{lemma}

\begin{proof}
 The construction of   $f$ will be more or less difficult depending on the arithmetical properties of $\alpha$. Let $(p_n/q_n)$ be the 
$n$th convergent of $\alpha$
 in its continued fraction expansion. For each $n\geq 0$, there exists $\tau_n\geq 1$ such that $q_{n+2}=q_n^{\tau_n}$. We define 
 $$\tau:=\liminf_n \tau_n\in [1,+\infty].$$
 We then fix $\nu'\in (0,1)$ such that $\nu'>\nu$, $\xi+\nu'<1$ and 
  \begin{equation}\label{*ksnue} 
 \sqrt{\frac{\xi}{1-\nu'}}<s.
 \end{equation}
 If moreover $\tau<\sqrt{\frac{1-\nu}\xi}$, we also require
 that $\tau<\sqrt{\frac{1-\nu'}\xi}$. 
 
 Let $n$ be a large integer and consider the partition of $\TT$ described in Section \ref{sec:partition}:
 $$\TT=\bigcup_{0\leq j<q_{n+1}}\Delta_j^{(n)}\cup \bigcup_{0\leq j<q_n}\Delta_j^{(n+1)}.$$
 Again for ease of notation we suppose that $q_{n}$ and $q_{n+1}$ are even;
 if not, a modification similar to the one used in the proof of Lemma \ref{lem:continuous}
 can be used.
 
 \noindent\textsc{First case: }$\tau\geq\sqrt{\frac{1-\nu}\xi}$. Then, for $n$ large enough, $\tau_n s>1+\eta$ for some fixed $\eta>0$. We fix such an $n$
 and we then define $f$ as follows:
 \begin{itemize}
  \item on $\Delta_j^{(n)}=(a_j,b_j)$, $0\leq j<\frac{q_{n+1}}2$, $f$ is equal to $(x-a_j)^\xi$ on $\left[a_j,\frac{a_j+b_j}2\right]$, equal to $(b_j-x)^\xi$
  on $\left[\frac{a_j+b_j}2,b_j\right]$.
  \item On $\Delta_j^{(n)}=(a_j,b_j)$, $ \frac{q_{n+1}}2
  \leq j<q_{n+1}$, $f$ is equal to $-(x-a_j)^\xi$ on $\left[a_j,\frac{a_j+b_j}2\right]$, equal to $-(b_j-x)^\xi$
  on $\left[\frac{a_j+b_j}2,b_j\right]$.
  \item $f$ is equal to $0$ otherwise.
 \end{itemize}
 It is then clear that $\|f\|_\infty\leq 1$, $\lip_\xi(f)\leq 1$ and $\int_\TT fd\lambda=0$.   Recalling that $d_{n}=b_{j}-a_{j}$ for $0\leq j <q_{n+1}$  we then set 
 $$\delta_0=\sqrt{\frac{\xi}{1-\nu'}}\in (0,1),$$
 $$\gamma_0=\frac1{\delta_0}=\sqrt{\frac{1-\nu'}{\xi}}>1,$$
 $$\Gamma_j=\left(a_j+d_n^{\gamma_0},b_j-d_n^{\gamma_0}\right), \ \  0\leq j <q_{n+1},$$
 $$E_0=\bigcup_{j=0}^{\frac{q_{n+1}}2-1- \lf q_{n+1}^{\delta_0}\rf }\Gamma_j\cup\bigcup_{j=\frac{q_{n+1}}2}^{q_{n+1}-1- \lf q_{n+1}^{\delta_0}\rf }\Gamma_j.$$
Observe that if $y\in\Gamma_j$, then $|f(y)|\geq d_n^{\gamma_0\xi}$
 and that $R_\alpha(\Gamma_j)\subset\Gamma_{j+1}$, $0\leq j<q_{n+1} -1$.
It follows that, for $x\in E_0$ 
with constants  $C$ which do not depend on $n$ and 
may change from line to line
\begin{eqnarray}
\label{*eoest}
 \left|S_{ \lf q_{n+1}^{\delta_0}\rf  ,\alpha}f(x)\right|&\geq&C  \lf q_{n+1}^{\delta_0}\rf  d_n^{\gamma_0\xi}  \\
 \nonumber
 &\geq&C q_{n+1}^{\delta_0}q_{n+1}^{-\gamma_0\xi}\\
 \nonumber
 &\geq&C q_{n+1}^{\delta_0\left(1-\frac{\gamma_0}{\delta_0}\xi\right)}\\
 \nonumber
&\geq &C q_{n+1}^{\delta_0 \nu'}\\
 \nonumber
&\geq &A\lf q_{n+1}^{\delta_0}\rf^{\nu}
\end{eqnarray}
{ provided} $n$ is large enough. Thus \eqref{eq:holder1} is satisfied with $m=\lf q_{n+1}^{\delta_0}\rf $ and $E=E_0$ for large values of  $n$. Moreover, $E_0^c$ is contained in the union of 
\begin{itemize}
 \item $2 \lf  q_{n+1}^{\delta_0}\rf +2$ intervals of size $d_n$ (the intervals $\Delta^{(n)}_j$ which are not considered);
 \item $2q_{n+1}$ intervals of size $d_n^{\gamma_0}$ (the extreme parts of the intervals $\Delta_j^{(n)}$);
 \item $q_n$ intervals of size $d_{n+1}$ (the intervals of the following generation $\Delta^{(n+1)}_j$).
\end{itemize}
Hence, for $n$ large enough,
$$\mathcal H_\delta^s(E^c)\leq C\left(q_{n+1}^{\delta_0}q_{n+1}^{-s}+q_{n+1}q_{n+1}^{-\gamma_0 s}+q_n q_n^{-\tau_n s}\right).$$
Since $\delta_0-s<0$, $1-\gamma_0 s<0$ and $1-\tau_n s<-\eta$, \eqref{eq:holder2} is also satisfied provided $n$ is large enough.

 \noindent\textsc{Second case: }$\tau<\sqrt{\frac{1-\nu}\xi}$. This time, the intervals coming from $\bigcup_j \Delta_j^{(n+1)}$ are too  long 
 to be neglected with respect to the $\mathcal H^s$-measure.   By  
the choice of $\nu'$, we know that there  exist  integers $n$ as large as we want
 such that
 \begin{equation}\label{*tne}
1\leq \sqrt{\tau_n}\leq \tau_n<\sqrt{\frac{1-\nu'}\xi};
 \end{equation}

 we will fix such an $n$ later. 
 We keep the same values for $\delta_0$, $\gamma_0$, $\Gamma_j$ and $E_0$ and the same definition for $f$ on $\bigcup_{0\leq j<q_{n+1}}\Delta_j^{(n)}$
  as in the first case. 
On the other hand, we
 define $f$ on $\Delta_j^{(n+1)}=( u_j ,v_j  )$ by imposing 
 $f(x)=(x- u_j )^\xi$ on $\left[ u_j ,\frac{ u_j + v_j  }2\right]$, $f(x)=( v_j  -x)^\xi$
  on $\left[\frac{ u_j + v_j  }2, v_j  \right]$ if  $0\leq j<q_n/2$ and  $f(x)=-(x- u_j)^\xi$ on $\left[ u_j ,\frac{ u_j + v_j  }2\right]$, $f(x)=-( v_j -x)^\xi$
  on $\left[\frac{ u_j + v_j }2, v_j  \right]$ if  $q_n/2\leq j<q_n$. We then set
 $$ \ddd_{1,n} = \delta_1=\sqrt{\tau_n}\sqrt{\frac{\xi}{1-\nu'}}\in (0,1) ,  $$
 $$ \ggg_{1,n} = \gamma_1=\frac1{\delta_1}=\frac1{\sqrt{\tau_n}}\sqrt{\frac{1-\nu'}{\xi}}>1 ,  $$
 $$\Theta_j=\left( u_j +d_{n+1}^{\gamma_1}, v_j -d_{n+1}^{\gamma_1}\right) ,  $$
 $$E_1=\bigcup_{j=0}^{\frac{q_n}2-1- \lf q_n^{\delta_1}\rf }\Theta_j\cup\bigcup_{j=\frac{q_n}2}^{q_n-1- \lf q_n^{\delta_1}\rf}\Theta_j ,  $$
 and $E=E_0\cup E_1$.  Remember that $d_{n+1}\approx q_n^{-\tau_n}$. We can still use \eqref{*eoest} and can deduce analogously for any $x\in E_1$,
 \begin{eqnarray}
  \label{*CA*} 
\left|S_{ \lf q_n^{\delta_1}\rf ,\alpha}f(x)\right|&\geq&C q_n^{\delta_1}d_{n+1}^{\gamma_1\xi}\\ 
\nonumber
  &\geq&C q_n^{\delta_1-\tau_n\gamma_1\xi}\\
  \nonumber
&\geq&C\lf q_n^{\delta_1}\rf^{\nu'}\\
 \nonumber
&\geq&A\lf q_n^{\delta_1}\rf^{\nu}
 \end{eqnarray}
provided $n$ is large enough.  From now on we can fix a sufficiently large $n$. The set $E^c$ consists  of   at most
 \begin{itemize}
 \item $2 \lf q_{n+1}^{\delta_0}\rf +2$ intervals of size $d_n$;
 \item $2q_{n+1}$ intervals of size $d_n^{\gamma_0}$;
 \item $2 \lf q_n^{\delta_1}\rf +2$ intervals of size $d_{n+1}$;
 \item $2q_n$ intervals of size $d_{n+1}^{\gamma_1}$.
\end{itemize}
Thus,
$$\mathcal H^s_\delta(E^c)\leq C\left(q_{n+1}^{\delta_0-s}+q_{n+1}^{1-\gamma_0 s}+q_n^{\delta_1-\tau_n s}+q_n^{1-\tau_n \gamma_1 s}\right).$$
By using \eqref{*ksnue} and \eqref{*tne} we conclude exactly as before  since 
$$\delta_1-\tau_n s\leq \tau_n \left(\sqrt{\frac{\xi}{1-\nu'}}-s\right)\leq\sqrt{\frac{1-\nu'}{\xi}}\left(\sqrt{\frac{\xi}{1-\nu'}}-s\right)<0$$
and
$$1-\tau_n \gamma_1 s\leq 1-\sqrt{\frac{1-\nu'}\xi}s<0.$$
\end{proof}

\begin{proof}[Proof of Theorem \ref{thm:holder}]
We will prove slightly more than announced. Let $\mathcal E^\xi$ be the closed subspace of $\mathcal C_0^\xi(\TT)$ defined by
$$\mathcal E^\xi=\left\{f\in\mathcal C_0(\TT);\ \forall x,y\in\TT,\ |f(x)-f(y)|\leq |x-y|^\xi\right\}=\left\{f\in\mathcal C_0(\TT);\ \textrm{Lip}_\xi(f)\leq 1\right\}.$$
The space $\mathcal E^\xi$, equipped with the norm of the uniform convergence is now again a separable complete metric space. We will prove that, for all functions $f$ in a residual subset of $\mathcal E^\xi$, for all $\nu\in (0,1-\xi)$, $\dimh\left(\mathcal F_\nu(f)\right)\leq \sqrt{\frac\xi{1-\nu}}$. 
Since $\mathcal F_\nu(f)\subset\mathcal F_{\widetilde{\nu}}(f)$ provided $\nu\leq\widetilde{\nu}$, it is sufficient to prove this inequality for $\nu$ belonging to a sequence $(\nu_k)$ which is dense
in $(0,1-\xi)$. Now, the countable intersection of residual sets remaining residual, we  just have to prove that, for a fixed $\nu\in(0,1-\xi)$, all functions $f$
in a residual subset of $\mathcal E^\xi$ satisfy $\dimh(\mathcal F_\nu)\leq\sqrt{\frac{\xi}{1-\nu}}$.

Let $(h_l)$ be a sequence of $\mathcal C_0$-coboundaries which is dense in $\mathcal E^\xi$ and with $\textrm{Lip}_\xi(h_l)\leq 1-\frac 1l$. For any $l\geq 1$,
there exists $C_l>0$ such that $\sup_n \left\|S_{n,\alpha}h_l\right\|_\infty\leq C_l$. Let $f_l$, $N_l$ and $E_l$ be given by Lemma \ref{lem:holder}
with $s=\sqrt{\frac{\xi}{1-\nu}}+\frac 1l$, $\delta=\veps=\frac 1l$, $M=l$ and $A=l(C_l+l+1)$. We set $g_l=h_l+\frac 1l f_l$ so that $g_l\in\mathcal E^\xi$, $(g_l)$
is dense in $\mathcal E^\xi$ and, for any $x$ in the compact set $E_l$, there exists $m\in \{l,\dots,N_l\}$ with
$$|S_{ m,\alpha}g_l(x)|\geq (l+C_l+1) m^\nu-C_l\geq (l+1) m^\nu.$$
 We can then find  $\delta_l>0$ such that, for all $f\in B(g_l,\delta_l)$ and all $x\in E_l$,  there exists $m\in \{l,\dots,N_l\}$ with
\[
|S_{ m,\alpha}f(x)|\geq l m^\nu  \text{ and } \mathcal H^s_{1/l}(E_l^{c})<\frac{1}{l}.
\]
We set $\mathcal R=\bigcap_{L\geq 1}\bigcup_{l\geq L}B(g_l,\delta_l)\cap\mathcal E^\xi$ which is a residual subset of $\mathcal E^\xi$. 
Pick $f\in\mathcal R$. There exists an increasing sequence $(l_k)$ such that $f\in B(g_{l_k},\delta_{l_k})$. We set $E=\limsup_{ k} E_{l_k}$ and observe that, 
for any $x\in E$, 
$$\limsup_n \frac{|S_{n,\alpha}f(x)|}{n^\nu}=+\infty$$
so that $\mathcal F_\nu\subset E^c$. Now the construction of the sets $E_l$ ensures that
$$\dimh(E^c)\leq\sqrt{\frac{\xi}{1-\nu}}.$$
\end{proof}

\begin{question}
Is the value $\sqrt{\frac{1-\xi}{\nu}}$ optimal? In particular, it does not depend on the type of $\alpha$, which may look surprizing.
\end{question}

\section{Miscellaneous remarks}
\subsection{Open questions}
Our study  suggests  further questions. The first one is related to Corollary \ref{cor:individualresults}.
\begin{question}
Does there exist $\nu\in [1/2,1]$ such that
\begin{enumerate}
\item for all $\gamma>\nu$, for all $f\in\mathcal C_0(G)$, 
$$\mu\left(\left\{u\in G;\ \limsup_n \frac{S_{n,u}f(0)}{n^\gamma}\geq 1\right\}\right)=0;$$
\item for all $\gamma<\nu$, there exists a residual subset $\mathcal R$ of $\mathcal C_0(G)$ such that, for all $f\in\mathcal R$, 
$$\mu\left(\left\{u\in G;\ \limsup_n \frac{S_{n,u}f(0)}{n^\gamma}=+\infty\right\}\right)=1?$$
\end{enumerate}
\end{question}

It can be shown that $\nu=1/2$ works for (ii). Indeed, Lemma \ref{lem:key} and Fubini's theorem imply that, for all $\veps\in(0,1)$, all $C>0$ and all $M\in\NN$, there exist $x\in G$,
 $f\in \mathcal C_0(G)$, $N>M$ and $E\subset G$ with $\|f\|_\infty<\veps$, $\mu(E)>1-\veps$ 
 and $\sup_{M\leq n\leq N}\frac{S_{n,u}f(x)}{n^{1/2}}\geq C$ for $u\in E$. Translating $f$ if necessary, we may assume that $x=0$. We then conclude exactly as in the proof of Theorem \ref{thm:measuretwovariables}.

Second, Theorem \ref{thm:continuous} improves Theorem \ref{thm:individualresults} for rotations of the circle by replacing nowhere dense sets with the more precise notion of 
sets with zero Hausdorff dimension.
There are also enhancements of meager sets, for instance $\sigma$-porous sets (see \cite{Za87})

\begin{question}
 Does there exist a residual subset $\mathcal R$ of $\mathcal C_0(\TT)$ such that, for any $f\in\mathcal R$, $\mathfrak E_\ppp(f)$ is $\sigma$-porous?
\end{question}

In the spirit of Theorem \ref{thm:measuretwovariables}, the next step would be to perform a multifractal analysis of the exceptional sets. Precisely, let $f\in\mathcal C_0(\TT)$
and $\nu\in(1/2,1)$. Let us set
$$\mathcal E^-(\nu,f)=\left\{ (\alpha,x)\in\TT^2;\ \limsup_n \frac{\log|S_{n,\alpha}f(x)|}{\log n}\geq\nu\right\}.$$
These sets have Lebesgue measure zero.
\begin{question}
Can we majorize the Hausdorff dimension of $\mathcal E^-(\nu,f)$?
\end{question}

We could also replace everywhere the $\limsup$ by $\liminf$. 
\begin{question}
Let $\psi:\NN\to\NN$ with $\psi(n)=o(n)$. Does there exist $f\in\mathcal C_0(\Omega)$ such that $\{x\in\Omega;\ \liminf_n |S_{n,T}f(x)|/|\psi(n)|=+\oo\}$ is residual? has full measure? 
\end{question}

\subsection{Other sums}
The study of $S_{n,\alpha}f(x)$ is a particular case of the series $\sum_{n\geq 1}a_n f(x+n\alpha)$.
In the particular case $a_n=1/n$ this series, also called the one-sided ergodic Hilbert transform, was
thoroughly investigated in \cite{FanSchme17}.

In \cite{FanSchme17}, the authors show that for any non-polynomial function $f\in\mathcal C^2_0(\TT)$ with values in $\mathbb R$, there exists a residual set
$\mathcal R_f$ of irrational numbers depending on $f$ such that, for every $\alpha\in\mathcal R_f$, 
$$\limsup_N \sum_{n=1}^N \frac{f(x+n\alpha)}{n}=+\infty$$
for almost every $x\in\TT$ and they ask if this holds for every $x\in\TT$ (they show that this is the case if $\hat f(n)=0$ when $n\leq 0$). We provide a counterexample.
\begin{example}
 Let $a\in(0,1)$ and $f\in\mathcal C^2_0(\mathbb T)$ be defined by its Fourier coefficients $\hat f(0)=0$, $\hat f(n)=ia^n$ for $n>0$, $\hat f(n)=-i a^{-n}$ for $n<0$. A small computation shows that
 $$f(x)=\frac{iae^{2\pi i x}}{1-ae^{2\pi ix}}-\frac{iae^{-2\pi ix}}{1-ae^{-2\pi i x}}=\frac{-2a\sin(2\pi x)}{1-2a\cos(2\pi x)+a^2}.$$
 We shall prove that the one-sided ergodic Hilbert transform of $f$ is bounded at $x=0$. Indeed, setting
 $$G_N(t)=\sum_{n=1}^N \frac{e^{2\pi i nt}}n,$$
 it is easy to show that
 \begin{eqnarray*}
  \sum_{n= 1}^N \frac{f(n\alpha)}n&=&\sum_{k>0}ia^k G_N(k\alpha)-\sum_{k>0}
   i  a^k G_N(-k\alpha)\\
  &=&i\sum_{k>0}a^k \big(G_N(k\alpha)-\overline{G_N(k\alpha)}\big)\\
  &=&  -  2\sum_{k>0}a^k \Im m \big(G_N(k\alpha)\big).
 \end{eqnarray*}
 Now, it is well-known that the imaginary part of $G_N(t)$, namely $\sum_{n=1}^N \frac{\sin(2\pi nt)}{n}$ is uniformly bounded in $N$ and $t$
 (see e.g. \cite[p.4]{KAHANEABS}).
\end{example}
\begin{question}
Can we investigate, in the spirit of this paper and of \cite{FanSchme17}, the case $a_n=n^{-a}$, with $0<a<1$?
\end{question}

\subsection{Coboundaries in $\mathcal C^\xi_0(\TT)$}
The natural norm in $\mathcal C^\xi_0(\TT)$ is given by 
\begin{equation}\label{*cksi}
\Vert f\Vert_\xi=\sup_{x\in\TT}{|f(x)|} +\sup_{\substack{x,y\in\TT\\ x\not=y}}
\frac{|f(x)-f(y)|}{|x-y|^{  \xi  }}.
\end{equation}

One may wonder whether, in Theorem \ref{thm:holder}, we have residuality in $(\mathcal C^\xi_0(\TT),\Vert\ \Vert_\xi)$ instead of in $(\mathcal E^\xi,\Vert\ \Vert_\infty)$. A natural way to do that would be to prove
that the coboundaries are dense in $\mathcal C^\xi_0(\TT)$.
This is not the case, which shows again that
$\mathcal C_0^\xi(\TT)$ is a weird space. 

In $\mathcal C_0^\xi(\TT)$ we denote the ball of radius $r$ centered at $f\in \mathcal C_0^\xi(\TT)$
by $B_0^\xi(f,r)$, that is $g\in B_0^\xi(f,r)$ if and only if $\Vert g-f\Vert_\xi<r$.
We shall prove the following precise statement.

\begin{theorem}\label{*THCB}
For any $\aaa\in\RR\sm \QQ$ for any $\xi\in (0,1)$
there exists $f\in \mathcal C_0^\xi(\TT)$  such that for any $g\in B_0^\xi(f,0.1)$
the function $g$ is not a 
$\mathcal C_{0}$ (and hence not a $\mathcal C_0^\xi$)-coboundary, that is
there is no $u\in \mathcal C_0(\TT)$  such that $g=u\circ R_{\aaa}-u$.
Hence $\mathcal C_{0}$-coboundaries are not dense in $\mathcal C_0^\xi(\TT)$.
\end{theorem}

\begin{proof}
By induction we select $n_{1}=1$, $n_{k}\in \NN$, $J_{k}\sse [n_{k},n_{k+1})\cap \ZZ$
with the following properties.
If we let $\ds h_{k}=\Big (\frac{k}{n_{k+1}}\Big )^{1/\xi}$ then the intervals
\begin{equation}\label{*CB1a}
\{ [j\aaa-h_{k},j\aaa+3h_{k}] ; j\in J_{k},\ k\in \NN \} \text{ are pairwise disjoint}
\end{equation}
(all these intervals are considered $\mod 1$ on $\TT$),
\begin{equation}\label{*CB1c}
\lll\left(\bigcup_{j\in J_{k}}[j\aaa-h_{k},j\aaa+3h_{k}]\right)<\frac{1}{100^{k+2}},
\end{equation}
\begin{equation}\label{*CB1d}
m_{k}\defeq \# J_{k}> 0.99\cdot n_{k+1}.
\end{equation}
For this property we can use that $\bigcup_{k'<k}\bigcup _{j\in J_{k'}}[j\aaa-h_{k'},j\aaa+3h_{k'}]$ is a union 
 of intervals, which by \eqref{*CB1c}  are of total measure less than $1/200$ and the sequence
$(j\aaa)$ is uniformly distributed on $\TT$, especially if we suppose that the
$n_{k}$s are denominators of suitable convergents of $\aaa$
and recall Subsection \ref{sec:partition}.
We also suppose that $J_{k}$ is maximal possible, by this we mean that
if $j\in [n_{k},n_{k+1})\cap \ZZ$ and $j\not \in J_{k}$ then 
\begin{equation}\label{*CB1e}
[j\aaa-h_{k},j\aaa+3h_{k}]\cap \bigcup_{k'<k}\bigcup _{j'\in J_{k'}}[j'\aaa-h_{k'},j'\aaa+3h_{k'}]\not=\ess.
\end{equation}

By the definition of $h_{k}$ and \eqref{*CB1d} we have 
\begin{equation}\label{*CB1b}
m_{k}\cdot h_{k}^{\xi}>0.99 \cdot k.
\end{equation}

Next we define $f$. On an interval $[j\aaa-h_{k},j\aaa+3h_{k}]$, $j\in J_{k}$, $k\in\NN$ we define
$f$ in the following way:
$f(j\aaa-h_{k})=f(j\aaa+h_{k})=f(j\aaa+3h_{k})=0$ and
\begin{equation}\label{*CB2b}
f(j\aaa)=h_k^\xi,\  f(j\aaa+2h_{k})=-h_k^\xi,
\end{equation}
otherwise $f$ is linear on each $[j\aaa +nh_{k},j\aaa+(n+1)h_{k}]$ with $n\in [-1,0,1,2]$.
If $x\not \in \cup_{k\in \NN}\cup _{j\in J_{k}}[j\aaa-h_{k},j\aaa+3h_{k}]$ then we set $f(x)=0$.

It is obvious that $f \in \mathcal C_0^\xi(\TT)$ with $\textrm{Lip}_\xi(f)\leq 1$. 

Suppose that $g\in B_0^\xi(f,0.1)$ and proceeding towards a contradiction suppose that
$g=u\circ R_{\aaa}-u$ with a $u\in \mathcal C_0(\TT).$
Then there exists $K_{u}$  such that $|u|\leq K_{u}$.

Clearly, for any $x\in\TT$ and any $n\in\mathbb N$, we have
\begin{equation}\label{*CB2a}
|S_{n,\aaa}g(x)|=\left|\sum_{j=0}^{n-1}g(x+j\aaa)\right|=|u(x+(n+1)\aaa)-u(x)|\leq 2 K_{u}.
\end{equation}

We will prove in \eqref{*CB3c} and \eqref{*CB3d} that for any function $g\in B_0^\xi(f,0.1)$, its Birkhoff sums are not bounded and this will provide a contradiction.

Suppose $k$ is fixed. Since $g\in B_0^\xi(f,0.1)$ we have for any $j\in J_{k}$
\begin{equation}\label{*CB3a}
\frac{|f(j\aaa)-g(j\aaa)-(f(j\aaa+2h_{k})-g(j\aaa+2h_{k}))|}{|2h_{k}|^{\xi}}<0.1.
\end{equation}

This and \eqref{*CB2b} imply that for $j\in J_{k}$
\begin{equation}\label{*CB3b}
g(j\aaa)-g(j\aaa + 2h_k )\geq 0.9\cdot 2^{\xi} h_{k}^{\xi}>0.9 h_{k}^{\xi}=0.9\frac{k}{n_{k+1}}.
\end{equation}

Next we consider the cases when $j\not\in J_{k}$, $j\in [n_{k},n_{k+1})$. Then 
\eqref{*CB1e} applies.
Suppose first that there exists $k'<k$, $j'\in J_{k'}$,  such that 
$j\aaa,j\aaa + 2h_k\in [j'\aaa-h_{k'},j'\aaa+3h_{k'}]$.
The construction of $f$ on $[j'\alpha-h_{k'},j'\alpha+3h_{k'}]$ ensures that
\begin{equation}\label{*CB4**a}
|f(j\aaa)-f(j\aaa + 2h_k)|\leq 2h_k (h_{k'})^{\xi-1}< 0.001\cdot h_{k}^{\xi}
\end{equation}
provided $n_{k+1}$ was choosen sufficiently large.

If $j\aaa\not\in \bigcup_{k'<k}\bigcup_{j'\in J_{k'}}[j'\aaa-h_{k'},j'\aaa+3h_{k'}]$ then either
$f(j\aaa)=0$, or $j\aaa\in \bigcup_{k'>k}\bigcup_{j'\in J_{k'}}[j'\aaa-h_{k'},j'\aaa+3h_{k'}]$.
In this latter case $|f(j\aaa)|\leq h_{k'}^{\xi}$ with $ k'> k$  and we can suppose by the inductive definition of the
$h_{k'}$ that $h_{k'}<0.0005^{1/\xi}\cdot h_{k} $.
Thus 
\begin{equation}\label{*CB5*a}
|f(j\aaa)|\leq 0.0005  \cdot h_k^\xi.
\end{equation}
Similarly if $j\aaa + 2h_k\not\in \bigcup_{k'<k}\bigcup_{j'\in J_{k'}}[j'\aaa-h_{k'},j'\aaa+3h_{k'}]$ we can suppose that
\begin{equation}\label{*CB5*b}
|f(j\aaa + 2h_k)|\leq 0.0005  \cdot h_k^\xi.
\end{equation}

In case one of $j\aaa,j\aaa + 2h_k$ belongs to 
a $[j'\aaa-h_{k'},j'\aaa+3h_{k'}]$, $k'<k$, $j'\in J_{k'}$ and the other is not an element of any such interval 
then $f(x)=0$ at some $x$ in $[j\aaa,j\aaa + 2h_k]$ and a combination
of \eqref{*CB4**a} and \eqref{*CB5*a}, or \eqref{*CB5*b}
is applicable.

Summarizing, we have finally shown that for all $j\in [n_k,n_{k+1})\backslash J_k$, 
\begin{equation}\label{*CB6*a}
|f(j\aaa)-f(j\aaa + 2h_k)|< 0.002  \cdot  h_k^\xi.
\end{equation}

Since $g\in B_0^\xi(f,0.1)$ by \eqref{*CB3a} and \eqref{*CB6*a}  we obtain  
\begin{equation}\label{*CB6*b}
|g(j\aaa)-g(j\aaa + 2h_k)|< 0.102\cdot h_k^\xi \cdot 2^{\xi}.
\end{equation}

We claim that either
\begin{equation}\label{*CB3c}
\sum_{j=n_{k}}^{n_{k+1}-1}
g(j\aaa)\geq \frac{1}{4} n_{k+1}h_k^\xi > \frac{k}{4}
\end{equation}
(see \eqref{*CB1b} as well), or
\begin{equation}\label{*CB3d}
\sum_{j=n_{k}}^{n_{k+1}-1}
g(j\aaa + 2h_k)\leq -\frac{1}{4} n_{k+1}h_k^\xi <- \frac{k}{4}.
\end{equation}
It is clear that for large $k$ this will contradict \eqref{*CB2a}. 

Next suppose that the negation of \eqref{*CB3c} and the negation
of \eqref{*CB3d}
hold.

This implies 
\begin{equation}\label{*CB7*a}
\sum_{j=n_{k}}^{n_{k+1}-1}
\left(g(j\aaa)-g(j\aaa + 2h_k)\right) < 2\cdot\frac{k}4=\frac{1}{2} n_{k+1}h_k^\xi.
\end{equation}

By \eqref{*CB3b} and \eqref{*CB1d}
\begin{equation}\label{*CB7*b}
\sum_{j\in J_{k}}\left(g(j\aaa)-g(j\aaa + 2h_k)\right)\geq \# J_{k} \cdot 0.9
\frac{k}{n_{k+1}}> 0.99\cdot n_{k+1}\cdot 0.9 h_k^\xi.
\end{equation}
On the other hand, by \eqref{*CB1d} and \eqref{*CB6*b}
\begin{equation}\label{*CB7*c}
\sum_{\substack{j=n_{k}\\ j\not\in J_{k}}}^{n_{k+1}-1}
|g(j\aaa)-g(j\aaa + 2h_k)| < 0.01\cdot  n_{k+1} \cdot 0.102\cdot h_k^\xi  \cdot 2^{\xi}.
\end{equation}
Now \eqref{*CB7*b} and \eqref{*CB7*c} contradict \eqref{*CB7*a}.
\end{proof}

\providecommand{\bysame}{\leavevmode\hbox to3em{\hrulefill}\thinspace}
\providecommand{\MR}{\relax\ifhmode\unskip\space\fi MR }
\providecommand{\MRhref}[2]{%
  \href{http://www.ams.org/mathscinet-getitem?mr=#1}{#2}
}
\providecommand{\href}[2]{#2}


\bibliographystyle{amsplain} 
\bibliography{bby} 

\end{document}